%% file: multmoore.tex
\pdfoutput=1\relax
\documentclass{amsart}

\input{Preamble.tex}
\usepackage[margin=1.5in]{geometry}

\newtoggle{draft}
\togglefalse{draft}
\iftoggle{draft} {
\newcommand{\NB}[1]{\todo[color=gray!40]{#1}}
\newcommand{\TODO}[1]{\todo[color=red]{#1}}
\usepackage{showkeys}
}{ % else
\newcommand{\NB}[1]{}
\newcommand{\TODO}[1]{}
\renewcommand{\todo}[1]{}
\renewcommand{\todo}[1]{}
}

\title{Multiplicative structures on Moore spectra}
\date{\today}

\author{Robert Burklund}
\address{Department of Mathematics, MIT, Cambridge, MA, USA}
\email{burklund@mit.edu}

\begin{document}

\begin{abstract}
  In this article we show that
  %in many situations the cofiber of a power of a map can be given an $\E_n$-algebra structure.
  %As a corollary we conclude that
  $\Ss/8$ is an $\E_1$-algebra,
  $\Ss/32$ is an $\E_2$-algebra,
  $\Ss/p^{n+1}$ is an $\E_n$-algebra at odd primes
  and, more generally,
  for every $h$ and $n$ there exist generalized Moore spectra of type $h$ which admit an $\E_n$-algebra structure.
\end{abstract}

\maketitle

% \setcounter{tocdepth}{1}
% \tableofcontents
\vbadness 5000

% ----------------------------------------------------------------------%

\input{intro.tex}

\subsection*{Acknowledgments}\

The author would like to thank 
Prasit Bhattacharya,
Mike Hopkins,
Ishan Levy,
Piotr Pstr\k{a}gowski and
Andy Senger for helpful conversations related to this work.
We would further like to thank Ishan Levy, Andy Senger, Haynes Miller and Allen Yuan for comments on a draft.
Special thanks goes to Shaul Barkan for pointing out several inadequacies with our original treatment of bar-cobar duality.

\input{body.tex}

\appendix

\input{barcobar.tex}

\bibliographystyle{alpha}
\bibliography{bibliography}

\end{document}

%% file: Preamble.tex
\usepackage{verbatim}
\usepackage[textsize=scriptsize]{todonotes}
\usepackage{tikz-cd}
\usepackage{etoolbox}
\usepackage{etex}
\usepackage[T1]{fontenc}
\usepackage{chemarr}
\usepackage{amssymb}
\usepackage{amsmath}
\usepackage{comment}
\usepackage{mathtools}
\usepackage{rotating}
\usepackage{wrapfig}
\usepackage{outlines}
\usepackage{graphicx}
\usepackage{scalerel}
\usepackage{bbm}
\usepackage{multicol}

\usepackage{amsthm}

\usepackage{spectralsequences}
\let\oldwidetilde\widetilde
\protected\def\widetilde{\oldwidetilde}

\usepackage{hyperref}
\usepackage{cleveref}

\hypersetup{
   colorlinks,
   linkcolor={red},
   citecolor={green!30!black},
   urlcolor={blue}
}

\usepackage{tikz}
\usetikzlibrary{matrix,arrows,decorations}
\usepackage{tikz-cd}

\usepackage{adjustbox}

\let\oldtocsection=\tocsection
 
\let\oldtocsubsection=\tocsubsection
 
\let\oldtocsubsubsection=\tocsubsubsection
 
\renewcommand{\tocsection}[2]{\hspace{0em}\oldtocsection{#1}{#2}}
\renewcommand{\tocsubsection}[2]{\hspace{1em}\oldtocsubsection{#1}{#2}}
\renewcommand{\tocsubsubsection}[2]{\hspace{2em}\oldtocsubsubsection{#1}{#2}}

\theoremstyle{definition}

\newtheorem{nul}{}[section]
\newtheorem{dfn}[nul]{Definition}

\newtheorem{rmk}[nul]{Remark}

\newtheorem{cnstr}[nul]{Construction}
\newtheorem{cnv}[nul]{Convention}

\newtheorem{exm}[nul]{Example}

\newtheorem{rec}[nul]{Recollection}

\newtheorem*{dfn*}{Definition}
\newtheorem*{axm*}{Axiom}
\newtheorem*{ntn*}{Notation}
\newtheorem*{exm*}{Example}
\newtheorem*{exr*}{Exercise}
\newtheorem*{int*}{Intuition}
\newtheorem*{qst*}{Question}
\newtheorem*{rmk*}{Remark}

\theoremstyle{plain}

\newtheorem{thm}[nul]{Theorem}
\newtheorem{prop}[nul]{Proposition}

\newtheorem{lem}[nul]{Lemma}

\newtheorem{cor}[nul]{Corollary}

\newtheorem*{thm*}{Theorem}
\newtheorem*{prop*}{Proposition}
\newtheorem*{cor*}{Corollary}
\newtheorem*{lem*}{Lemma}
\newtheorem*{cnj*}{Conjecture}

\DeclareMathOperator*{\colim}{colim}

\DeclareMathOperator{\cof}{cof}

\DeclareMathOperator{\fib}{fib}

\DeclareMathOperator{\Map}{Map}

\DeclareMathOperator{\Fun}{Fun}

\DeclareMathOperator{\Mod}{Mod}

\def\Ind{\mathrm{Ind}}
\def\Pro{\mathrm{Pro}}

\def\Bar{\mathrm{Bar}}
\def\Cobar{\mathrm{Cobar}}
\def\Barn{\mathrm{Bar}^{(n)}}
\def\Cobarn{\mathrm{Cobar}^{(n)}}

\def\A{\mathbb{A}}

\def\E{\mathbb{E}}
\def\F{\mathbb{F}}

\def\R{\mathbb{R}}
\def\Ss{\mathbb{S}}
\def\Z{\mathbb{Z}}

\def\CC{\mathcal{C}}

\def\EE{\mathcal{E}}
\def\II{\mathcal{I}}

\def\gl1{\mathrm{gl}_1}

\def\o{\mathbbm{1}}

\def\Alg{\mathrm{Alg}}
\def\cAlg{\mathrm{coAlg}}

\def\Sp{\mathrm{Sp}}

\def\Syn{\mathrm{Syn}}

\def\Def{\mathrm{Def}}

\def\op{\mathrm{op}}
\def\Gr{\mathbf{Gr}}
\def\gr{\mathrm{gr}}
\def\Fil{\mathbf{Fil}}

\def\aug{\mathrm{aug}}

\newcommand{\pushout}{\arrow[ul, phantom, "\ulcorner", very near start]}

\newcommand{\wt}{\widetilde}

\def\Id{\mathrm{Id}}

\newcommand\xqed[1]{%
  \leavevmode\unskip\penalty9999 \hbox{}\nobreak\hfill
  \quad\hbox{#1}}
\newcommand\tqed{\xqed{$\triangleleft$}}

%% file: intro.tex
\section{Introduction}

A fundamental difference between ordinary commutative algebra and higher algebra is the difficulty of constructing quotient algebras. As an example, the quotient of $\Z$ by $2$ is the commutative algebra $\F_2$, while the quotient of the sphere spectrum by $2$ does not even admit a unital multiplication.\footnote{A unital multiplication on $\Ss/2$ can be ruled out by considering the Cartan formula for Steenrod squares.} At odd primes $\Ss/p$ admits a unital multiplication, but Toda showed in 1968 that there is no homotopy associative multiplication on $\Ss/3$ \cite{TodaA3}. The issues continue at larger primes too; work of Kraines \cite{Kraines} and Kochman \cite{Kochman} implies\footnote{See \cite[Example 3.3]{Ang} for an explanation of this implication.} that $\Ss/p$ admits an $\A_{p-1}$-algebra structure, but no $\A_p$-algebra structure. In particular, $\Ss/p$ is never an $\E_1$-algebra.

The situation improves somewhat if we look at the quotient by a power of $p$ instead.
In 1984, Oka constructed homotopy associative multiplications on $\Ss/4$ and $\Ss/9$ \cite{OkaMult} and much more recently, Bhattacharya and Kitchloo constructed an $\A_{p^q-1}$-algebra structure on $\Ss/p^q$ for $p$ odd (on $\Ss/p^{q+1}$ for $p=2$) \cite{BhaMoore, BhaKit}.
Ultimately, the paucity of positive results has led to a rather negative outlook on the issue of quotient algebras---an outlook which led Mark Mahowald to conjecture that for $\alpha \in \pi_*\Ss$ the quotient $\Ss/\alpha$ admits an $\E_1$-algebra structure if and only if $\alpha = 0$.
In this article we show that this conjecture is completely false.

\begin{thm} \label{thm:E1-Moore}
  The Moore spectrum $\Ss/8$ admits an $\E_1$-algebra structure.
  
  At odd primes $\Ss/p^2$ admits an $\E_1$-algebra structure.
\end{thm}

With an $\E_1$-algebra structure on $\Ss/8$ in hand we are led to ask for more.
Could there be Moore spectra which are $\E_2$-algebras?
The answer, rather shockingly, is yes.

% Homotopy commutativity is automatic at odd primes and Oka constructed a homotopy commutative multiplication on $\Ss/8$ in \cite{OkaMult}.

\begin{thm} \label{thm:mult-Moore}
  Moore spectra admit the following multiplicative structures:
  \begin{itemize}
  % \item $\Ss/8$ is an $\E_1$-algebra,
  % \item $\Ss/32$ is an $\E_2$-algebra,
  \item $\Ss/2^{q}$ admits an $\E_n$-algebra structure for $q \geq \frac{3}{2}(n+1)$ and
  \item $\Ss/p^{q}$ admits an $\E_n$-algebra structure for $q \geq n+1$ and $p$ odd.
  \end{itemize}  
\end{thm}

% Pushing forward we are led to ask for even more.
% Chromatic homotopy theory provides us with a stratification of the category of finite spectra by height and
The Moore spectra we have considered up to now are only the first examples of the generalized Moore spectra obtained by inductively taking quotients by higher chromatic self-maps and multiplicative structures on these objects have received a certain amount of attention (see \cite{OkaSmall, OkaGrowth, DevSmall, DevBig}).
% Devinatz has shown that there exist generalized Moore spectra which admit a unital multiplication  and has analyzed the issue of commutativity in arity 2 in \cite{DevBig}.
Here too we are able to exceed all expectations.

%Going further with applications of \Cref{thm:En-quo} we obtain the following corollaries:

\begin{thm} \label{thm:gen-Moore}
  For each $h$ and $n$ there exist generalized Moore spectra
  % $\Ss/(p^{i_0},\dots,v_{h-1}^{i_{h-1}})$
  of type $h$ which admit an $\E_n$-algebra structure.
\end{thm}

Each of the theorems we have stated up to this point is proved as a corollary\footnote{Except for the $2$-primary parts of Theorems \ref{thm:E1-Moore} and \ref{thm:mult-Moore} which are proved separately in \Cref{sec:even-Moore}.} of the next theorem which allows us to construct quotient algebras in great generality.

\begin{thm} \label{thm:ring-quo}
  Suppose we are given an $\E_{m+1}$-algebra $A \in \Sp$ with $m \geq 2$
  and a class $v \in \pi_{2*}(A)$ such that
  \begin{itemize}
  \item $Q_1(v) \equiv 0 \pmod{v}$\footnote{Here $Q_1$ is the power operation carried by the second cell of $\mathrm{D}_2^{\E_m}$ on an even sphere. Note that since $2Q_1(v)=0$ this condition is automatic if $2$ is invertible in $R$. Similarly, $Q_1$ satisfies a Cartan formula and therefore this condition is satisfied by squares of even degree classes (see \Cref{rmk:sq-fine}).} or equivalently
  \item the cofiber $A/v$ admits a unital multiplication.
  \end{itemize}  
  Then, $A/v^{q}$ admits an $\E_n$-$A$-algebra structure as long as $n \leq m$ and $q > n$.
\end{thm}

This theorem too follows from an even more general result on quotients in higher algebra where we allow the category in which we work to vary.

% Theorems \ref{thm:mult-Moore}, \ref{thm:gen-Moore} and \ref{thm:ring-quo} are proved as applications\footnote{Except for the $p=2$ parts of \Cref{thm:mult-Moore} which is proved in seperately \Cref{sec:even-Moore}.} of our main theorem, which produces $\E_n$-algebra structures on quotients in great generality.

\begin{thm} \label{thm:En-quo}
  Suppose we are given a
  stably $\E_m$-monoidal\footnote{Meaning that we ask that the binary tensor product $- \otimes -$ on $\CC$ commutes with finite (co)limits separately in each variable.} category $\CC$ with $m \geq 2$  
  an object $\II \in \CC$ and a map $v: \II \to \o_\CC$
  such that the cofiber $\o_\CC/v$ admits a right unital multiplication.
  Then for each $n \leq m$ there exists a tower of $\E_n$-algebras and $\E_n$-algebra maps
  \[ \cdots \longrightarrow \o_\CC/v^{n+3} \longrightarrow \o_\CC/v^{n+2} \longrightarrow \o_\CC/v^{n+1}. \]
  Moreover, each $\o_\CC/v^q$ has a unique
  $v$-compatible (in the sense of \Cref{dfn:v-compatible} below) $\E_n$-algebra structure.
  %% Moreover, the $\E_n$-algebra structure we place on $\o_\CC/v^q$
  %% is the unique one which is $v$-compatible in the sense of \Cref{dfn:v-compatible} below.
\end{thm}

\begin{rmk}\label{rmk:force-pres}
  If we let $\CC'$ denote the full subcategory of $\CC$ generated by $\o_\CC$ and $\II$ under tensor products and finite (co)limits, then there is an $\E_m$-monoidal left adjoint
  $ \Ind(\CC') \to \CC $
  which is fully faithful on $\CC'$.

  This allows us to reduce the proof of \Cref{thm:En-quo} to the case where $\CC$ is
  stable, presentably $\E_m$-monoidal and $\o_\CC$ and $\II$ are compact $\otimes$-generators.
  \tqed
\end{rmk}

The proof of \Cref{thm:En-quo} has two main inputs:
An obstruction theory for constructing $\E_n$-algebra structures on quotients which we develop in \Cref{sec:budget} and
a categorification of the Adams spectral sequence constructed by Patchkoria and Pstr\k{a}gowski in \cite{PP21} which we review in \Cref{sec:Bockstein}.
The proof of \Cref{thm:En-quo} is then quite direct:
We construct a deformation of $\CC$ which categorifies the $\o_\CC/v$-Adams spectral sequence and an object $\nu\o_\CC/\wt{v}^q$ lying over $\o_\CC/v^{q}$. Then we compute that the groups in which the obstructions to constructing an $\E_n$-algebra structure on $\nu\o_\CC/\wt{v}^q$ live vanish!

% \vspace{0.3cm}
The proofs of the $2$-primary parts of Theorems \ref{thm:E1-Moore} and \ref{thm:mult-Moore} are simpler than (and somewhat different from) the proof of \Cref{thm:En-quo} and are given in \Cref{sec:even-Moore}.
In \Cref{sec:barcobar}, which may be of independent interest, we discuss bar-cobar duality for graded $\E_n$-algebras and use this duality to make the key construction used in \Cref{sec:budget}.

The reader might naturally wonder whether our results are sharp.
The following example shows that \Cref{thm:ring-quo} at least cannot be too far from sharp.

\begin{exm}
  Let $R$ be the free commutative $\F_2$-algebra on a class $x$ in degree zero.
  From \Cref{thm:ring-quo} we know that $R/x^{2(2^{k}-1)}$ admits an $\E_{2^k-2}$-algebra structure.
  On the other hand
  \[ Q_{2^k}(x^{2^{k+1} - 2}) \equiv x^{2(2^k -2)} Q_1(x)^{2^k} \not\equiv 0 \pmod{x^{2^{k+1} - 2}}, \]
  which implies that $R/x^{2(2^{k}-1)}$ is not an $\E_{2^k+1}$-$R$-algebra.
  \tqed
\end{exm}

The issue of whether \Cref{thm:mult-Moore} is sharp is more delicate.
We suspect it is sharp at odd primes, but that at $p=2$ the function $\frac{3}{2}(n+1)$ can be replaced by $n + O(1)$.

\subsection*{Conventions}

\begin{enumerate}
\item We fix three nested universes, referring to objects as small, large or huge depending on their size.
\item We write $\o_{\mathcal{E}}$ for the unit of a monoidal category $\mathcal{E}$.
\item We write $\o_{\mathcal{E}}\{ X \}$ for the free $\E_n$-algebra on an object $X \in \mathcal{E}$.
\item Throughout the body of the paper $\CC$ will be a fixed
  stable, presentably $\E_m$-monoidal category
  with compact generators $\{ \II^{\otimes k} \}_{k \geq 0}$ and $m \geq 2$.
\item We write $\CC^\Gr$ for the $\E_m$-monoidal category of graded objects in $\CC$.
  In $\CC^{\Gr}$ we write $X(1)$ for the shift of $X$ by $1$ and
  $X_k$ for the degree $k$ component of $X$.
\item We write $\CC^\Fil$ for the $\E_m$-monoidal category of filtered objects in $\CC$ and we take the convention that all filtrations are increasing.
  We write $\tau$ for the shift map $\tau : X(1) \to X$ on a filtered object.
\item We view graded objects as modules over the cofiber of the shift map $\tau$ in filtered objects and identify the associated graded functor with taking the cofiber by $\tau$.
\item We say an $\E_n$-algebra or coalgebra in $\CC^\Gr$ (or $\CC^\Fil$) is positively graded if it vanishes in negative gradings and is given by the unit in degree $0$.
\item Part of the data of a $\E_n$-monoidal structure on $\CC$ includes a $k$-fold tensor product functor
  $ \E_n(k) \times_{\Sigma_k} \CC^{\times k} \to \CC $.
  Precomposing with the diagonal gives a functor
  $ \CC \to \Fun( \E_n(k)_{h\Sigma_k}, \CC) $
  and we write $\mathrm{D}_k^{c\E_n}(-)$ for the functor
  $\CC \to \CC$ obtained by taking the limit over $\E_n(k)_{h\Sigma_k}$ (the space of unorderd configurations of $k$ points in $\R^n$).
\item When working with $\F_2$-synthetic spectra we use the convention that
  the $k$ index of $\Sigma^{k,s}$ is the topological degree and
  the $s$-index is Adams filtration.
  This has the pleasant feature that $(k,s)$ corresponds to
  $(x,y)$-coordinates in Adams spectral sequence charts.
  %Unfortunately this means that our conventions do note agree with those of \cite{boundaries}.
\end{enumerate}

%% file: body.tex
\section{An obstruction theory for quotients} \label{sec:budget}

In order to motivate the developments of this section let us consider an example.

\begin{exm} \label{exm:E1}
  Suppose we are given an object $\II$ in $\CC$ and a map $v : \II \to \o_\CC$ and we want to give the
  cofiber $\o_\CC/v$ an $\E_1$-algebra structure.

  A natural first step in constructing such an $\E_1$-algebra is to consider the $\E_1$-cofiber of $v$ which we write $\o_\CC/\!\!/v$.\footnote{By this we mean the pushout of the span $\o_\CC \leftarrow \o_\CC\{\II\} \to \o_\CC$ where $\II$ maps to the unit by $v$ on the left and by zero on the right.} What we would like to do is continue attaching further $\E_1$-cells to $\o_\CC/\!\!/v$ in order to eliminate the difference between $\o_\CC/v$ and $\o_\CC/\!\!/v$. In order to organize this procedure we pass to the filtered setting where we consider the $\E_1$-cofiber $\o_{\CC^\Fil}/\!\!/\tau v$ which we depict below
  \[ \cdots \longrightarrow 0 \longrightarrow \o_\CC \longrightarrow \o_\CC/v \longrightarrow (?) \longrightarrow (?) \longrightarrow \cdots. \]
  
  Through degree $1$ this is what we would expect to see in an $\E_1$-algebra structure on $\o_{\CC^\Fil}/\tau v$, but $\o_{\CC^\Fil}/\!\!/\tau v$ begins to deviate from $\o_{\CC^\Fil}/\tau v$ in degree $2$.
  To analyze this deviation we note that since the associated graded functor is $\E_m$-monoidal
  $\gr_\bullet \left( \o_{\CC^\Fil}/\!\!/\tau v \right) $ is a free graded $\E_1$-algebra on a copy of $\Sigma \II$ in degree $1$.
  This means that in degree $2$ our algebra $\o_{\CC^\Fil}/\!\!/\tau v$ has an extra copy of $\Sigma^2 \II^{\otimes 2}$ which we would like to eliminate. To do this we need to produce a lift
  \begin{center}
    \begin{tikzcd}
      & \o_{\CC^\Fil}/\!\!/\tau v \ar[d] \\
      \Sigma^2 \II^{\otimes 2}(2) \ar[r, hook] \ar[ur, dashed] & (\o_{\CC^\Fil}/\!\!/\tau v)/\tau
    \end{tikzcd}
  \end{center}
  which we can use to take a further $\E_1$-cofiber.
  In this manner we encounter our first obstruction:
  \[ \overline{Q_1}(v) \in [\Sigma^1 \II^{\otimes 2}, \o_\CC/v].\footnote{The reader can take this as the definition of $\overline{Q_1}(v)$.} \]
  If $\overline{Q_1}(v)$ vanishes, then upon taking the $\E_1$-cofiber of the associated lift we obtain a filtered $\E_1$-algebra $R$ which agrees with $\o_\CC/\tau v$ in degrees $0$, $1$ and $2$, but begins to deviate in degree $3$.  
  % Continuing this procedure, we would encounter each of the $\E_1$-cells of a square zero algebra on $\Sigma\LL(1)$ (the associated graded of the algebra we are constructing). Through bar-cobar duality we can identify these cells with the graded pieces of the cofree $\E_1$-coalgebra on $\Sigma^2\LL(1)$.
  % Thus, we expect to see a natural number indexed collection of obstructions each of which is a map out of some power of $\LL$. 
  \tqed
\end{exm}

Our goal in this section is to extend the manipulations of \Cref{exm:E1}
into an obstruction theory for constructing an $\E_n$-algebra structure on $\o_\CC/v$.
Before that we make a short digression on $\A_2$-structures.

\begin{lem}[{cf. \cite[p.27-28]{TodaBook}, \cite{GrayProduct} and \cite{OkaSmall}}] \label{lem:A2}
  In the situation described above, if $\overline{Q_1}(v)$ vanishes,
  then $\o_\CC/v$ admits a unital multiplication.
  %If the Euler characteristic of the line bundle $\LL$ is $1$ we can identify $\overline{Q_1(v)}$ with the image of the $\E_2$-power operation $Q_1(v)$ under the map $\o_\CC \to \o_\CC/v$.
\end{lem}

\begin{proof}
  % If $\overline{Q_1}(v)$ vanishes, then upon taking the $\E_1$-cofiber of the associated lift we have a filtered $\E_1$-algebra $R$ whose filtration $1$ and $2$ pieces are both $\o_\CC/v$.
  The desired unital multiplication appears as the degree $\leq 2$ component of the multiplication on the filtered $\E_1$-algebra $R$ constructed in \Cref{exm:E1}.
  % We can write the $\E_1$-cofiber $\o_{\CC^{\Fil}}/\!\!/\tau v$ as a relative tensor product
  % $\o_{\CC^{\Fil}} \otimes_{\o_{\CC^{\Fil}}\{\LL(1)\}} \o_{\CC^{\Fil}}$
  % over a free $\E_2$-algebra where $\LL(1)$ goes to the unit by zero on the left and by $\tau v$ on the right. The arity $2$ piece of $\o_{\CC}\{\LL\}$ can be identified with the cofiber of $1 - \chi(\LL)$ and therefore if the Euler characteristic of $\LL$ is $1$ we can pick a splitting of this arity two piece as $\LL^{\otimes 2} \oplus \Sigma \LL^{\otimes 2}$. The $\E_2$-power operation $Q_1$ is carried by the top cell of the arity two piece of this free algebra and in this way we can identify its image in $\o_\CC/v$ with the obstruction $\overline{Q_1(v)}$.
\end{proof}

\begin{lem} \label{lem:sqz-res}
  Given an $X \in \CC$
  there exists a sequence of $\E_n$-algebras in $\CC^{\Gr}$ converging to the trivial square zero extension of $\o_{\CC^{\Gr}}$ by a copy of $\Sigma X$ placed in degree $1$
  \[ \o_{\CC^{\Gr}} = R^0 \xrightarrow{r_1} R^1 \xrightarrow{r_2} R^2 \to \cdots \to \o_{\CC^{\Gr}} \oplus \Sigma X(1) \]
  such that
  \begin{enumerate}
  \item $R^{k}$ is equivalent to $\o_{\CC^{\Gr}} \oplus \Sigma X(1)$ through degree $k$.
  \item The map $r_{k}$ fits into a pushout square of $\E_n$-algebras
    \begin{center}
      \begin{tikzcd}
        \o_{\CC^{\Gr}}\left\{ \Sigma^{-1-n} \mathrm{D}_{k}^{c\E_n} \left( \Sigma^{n+1} X(1) \right) \right\} \ar[r, "\mathrm{aug}"] \ar[d, "s_{k}"] & \o_{\CC^{\Gr}} \ar[d] \\
        R^{k-1} \ar[r, "r_{k}"] & R^{k}. \pushout
      \end{tikzcd}
    \end{center}
  \end{enumerate}  
\end{lem}

% The proof of \Cref{lem:sqz-res} relies on Lurie's Koszul duality adjunction for $\E_n$-algebras in a key way so before proceeding we briefly review some salient points.
% In \cite[\S 5.2.3]{HA} Lurie constructs an adjunction
% \[ \Barn : \Alg_{\E_n}^{\aug}(\CC^{\Gr}) \rightleftharpoons \mathrm{coAlg}_{\E_n}^{\aug}(\CC^{\Gr}) : \Cobarn \]
% (see \cite[5.2.3.6 and 5.2.3.9]{HA} specifically).
% This bar-cobar adjunction has three main properties which we will need:
% \begin{enumerate}
% \item[(a)] $\Barn$ sends a free algebra on an object $Y$ to the square zero coalgebra $\o \oplus \Sigma^n Y$ \cite[5.2.3.15]{HA}.
% \item[(b)] Dually, $\Cobarn$ sends a cofree coalgebra on an object $Y$ to the square zero algebra $\o \oplus \Sigma^{-n}Y$.
% \item[(c)] $\Barn$ is computed by an $n$-fold bar construction \cite[5.2.3.14]{HA}.
%   In particular, this implies that in the positively graded setting,
%   if two $\E_n$-algebras $A$ and $B$ agree through grading $k$,
%   then $\Barn(A)$ and $\Barn(B)$ agree through grading $k$ as well.
% \end{enumerate}

\begin{proof}
  In proving this lemma we make heavy use of the material from \Cref{sec:barcobar}.
  The resolution is from \Cref{cnstr:free-res}.
  \Cref{lem:free-res-terms} identifies the objects $X^k$ from this construction
  with $\Sigma^{-1-n}\Barn(\o \oplus \Sigma X(1))_k$.
  Then, the combination of \Cref{thm:En-bar-cobar} and \Cref{lem:cofree-to-sqz} give us equivalences
  \[ \Barn(\o \oplus \Sigma X(1))_k \simeq \mathrm{coFree}(\Sigma^{n+1}X(1))_k \simeq \mathrm{D}_k^{c\E_n}(\Sigma^{n+1}X). \]
  \qedhere

\end{proof}

\begin{prop} \label{prop:obs}
  Given a map $r: X \to \o_\CC$ in $\CC$
  there exists a sequence of inductively defined obstructions
  \[ \theta_k \in [\Sigma^{-2-n}\mathrm{D}_k^{c\E_n}(\Sigma^{n+1}X) ,\ \ \o_\CC/r ] \qquad \text{ for  } \quad k \geq 2\]
  whose vanishing allows us to inductively construct a sequence of $\E_n$-algebras
  \[ \o_\CC = \overline{R}^0 \xrightarrow{\overline{r}_1} \overline{R}^1 \xrightarrow{\overline{r}_1} \overline{R}^2 \to \cdots \to \o_\CC/r \]
  converging to an $\E_n$-algebra structure on $\o_\CC/r$,
  where each map $\overline{r}_k$ sits in a pushout square
  \begin{center}
    \begin{tikzcd}
      \o_{\CC}\left\{ \Sigma^{-1-n} \mathrm{D}_{k}^{c\E_n} \left( \Sigma^{n+1} X(1) \right) \right\} \ar[r, "\mathrm{aug}"] \ar[d, "\overline{s}_{k}"] & \o_{\CC} \ar[d] \\
      \overline{R}^{k-1} \ar[r, "\overline{r}_{k}"] & \overline{R}^{k}. \pushout
    \end{tikzcd}
  \end{center}
  % $ \o_\CC/r $ can be given the structure of an $\E_n$-algebra.
\end{prop}

\begin{proof}
  In order to construct an $\E_n$-algebra structure on $\o_\CC/r$ we will construct an $\E_n$-algebra structure on the filtered object $\cof( X(1) \xrightarrow{\tau r} \o_{\CC^\Fil})$ whose associated graded is equipped with an equivalence with the square zero $\E_n$-algebra $\o_{\CC^\Gr} \oplus \Sigma X(1)$.
  The desired $\E_n$-algebra is then obtained by inverting the filtration parameter $\tau$.
  
  Let $R^k$ denote the graded $\E_n$-algebra from \Cref{lem:sqz-res}.
  We will inductively produce filtered $\E_n$-algebras $\wt{R}^k$ whose associated graded is $R^k$ and whose underlying $\E_n$-algebra is the desired $\overline{R}^k$.
  For our base case we let $\wt{R}^1$ be the $\E_n$-cofiber of the map $\tau r : X(1) \to \o_{\CC^\Fil}$.
  On associated graded this is a free algebra on $\Sigma X(1)$ and we pick an identification of it with $R^1$. Meanwhile, in degree $1$ we have $(\wt{R}^1)_1 \simeq \o_\CC/r$ as desired.

  Examining the pushout square from \Cref{lem:sqz-res}(2) we see that given a choice of $\wt{R}^{k-1}$
  and a lift  
  \begin{center}
    \begin{tikzcd}
      & \wt{R}^{k-1} \ar[d] \\      
      \o_{\CC^\Fil}\left\{ \Sigma^{-1-n} \mathrm{D}_k^{c\E_n} \left( \Sigma^{n+1} X(1) \right) \right\} \ar[r, "s_k"] \ar[ur, dashed, "\wt{s}_k"] &
      R^{k-1}
    \end{tikzcd}
  \end{center}
  we can construct an $\wt{R}^k$ as the pushout of $\wt{s}_k$ with the augmentation.
  Since we are mapping out of a free algebra and $R^{k-1} \cong \wt{R}^{k-1}/\tau$ the obstruction to constructing the lift $\wt{s}_k$ is the composite of $s_k$ with the $\tau$-Bockstein
  \[ \theta_k : \Sigma^{-1-n} \mathrm{D}_k^{c\E_n} \left( \Sigma^{n+1} X(1) \right) \xrightarrow{s_k} R^{k-1} \xrightarrow{\delta_{\tau}} \Sigma\wt{R}^{k-1}(1). \]

  Now we note that since $R^{k-1}$ vanishes in degrees $2,\dots,k-1$ by \Cref{lem:sqz-res}(1) we can identify $(\wt{R}^{k-1})_{k-1}$ with $\o_\CC/r$. Consequently, the obstruction $\theta_k$ lives in
  \[ [ \Sigma^{-2-n} \mathrm{D}_k^{c\E_n} \left( \Sigma^{n+1} X \right),\ \  \o_\CC/r ]. \]
\end{proof}

\begin{rmk} \label{rmk:uniqueness}
  The obstruction $\theta_k$ depends upon a choice of nullhomotopy of each of the previous obstructions. In particular, if $[\Sigma^{-1-n}\mathrm{D}_k^{c\E_n}(\Sigma^{n+1}X) ,\ \ \o_\CC/r ] = 0$, then there is at most one such nullhomotopy and therefore \Cref{prop:obs} produces at most one $\E_n$-algebra structure on the quotient.
  \tqed
\end{rmk}

As computing maps out of $\mathrm{D}_k^{c\E_n}(\Sigma^{n+1}X)$ can be unwieldy in general,
we end the section by noting that there is a simple resolution which lets us subdivide the obstruction $\theta_k$ into a collection of obstructions $\theta_{k,\alpha}$ whose sources are suspensions of $X^{\otimes k}$.

\begin{lem} \label{lem:Dn-fil}
  $\mathrm{D}_k^{c\E_n}(Y)$ has a resolution by finitely many copies of
  $\Sigma^{-c} Y^{\otimes k}$ where $0 \leq  c \leq (n-1)(k-1)$.
\end{lem}

\begin{proof}
  We can describe $\mathrm{D}_k^{c\E_n}(Y)$ as the limit over the diagram
  \[ F : \E_n(k)_{h\Sigma_k} \to \CC \]
  describing all ways of taking the $k^{\mathrm{th}}$ power of $Y$.
  The space $\E_n(k)_{h\Sigma_k}$ is equivalent to
  the space of unordered configurations of $k$ points in $\R^n$ and
  we can put a finite cellular filtration on this object whose cells lie in the range $0,\dots,(n-1)(k-1)$.\footnote{For an explicit presentation of this space as a finite simplicial set of dimension $(n-1)(k-1)$ one can take the quotient of the (free) $\Sigma_k$ action on the nerve of the poset of Fox--Neuwirth cells (see \cite{AyalaHepworth}).}
  On the limit of the diagram $F$ of shape $\E_n(k)_{h\Sigma_k}$ this induces a finite filtration whose associated graded is given by copies of $\Sigma^{-c} Y^{\otimes k}$ where $0 \leq c \leq (n-1)(k-1)$.
\end{proof}

\begin{cor} \label{cor:obs-degs}
  In the situation of \Cref{prop:obs} we can use 
  the filtration on $\mathrm{D}_k^{c\E_n}(-)$ from \Cref{lem:Dn-fil}
  to refine the obstructions $\theta_k$ to obstructions
  \[ \theta_{k,\alpha} \in [ \Sigma^{-2-n-c_\alpha}(\Sigma^{n+1}X)^{\otimes k},\ \ \o_\CC/r ] \]
  where $k \geq 2$ and $0 \leq c_\alpha \leq (n-1)(k-1)$.
\end{cor}

\begin{rmk} \label{rmk:E1-explicit}
  In the case $n=1$ the space $\E_1(k)_{h\Sigma_k}$ is a single point and therefore we have obstructions
  \[ \theta_k \in [ \Sigma^{-3}(\Sigma^{2}X)^{\otimes k},\ \ \o_\CC/r ]. \]
  \tqed
\end{rmk}

\section{$\E_n$-algebra Moore spectra} \label{sec:even-Moore}

The key idea in applying \Cref{prop:obs} to construct multiplicative structures on Moore spectra is that while the obstruction groups do not vanish in $\Sp$, they do vanish in certain deformations of $\Sp$. In this section we apply this idea using the category of $\F_2$-synthetic spectra as our deformation.

\begin{thm} \label{thm:mod8}
  $\Ss/8$ admits the structure of an $\E_1$-algebra.
\end{thm}

\begin{proof}
  In order to prove that $\Ss/8$ admits an $\E_1$-algebra structure we will show that in the $\F_2$-synthetic category $\nu\Ss/\wt{2}^3$ admits an $\E_1$-algebra structure and then invert $\tau$.

  Applying the obstruction theory from \Cref{prop:obs}
  with the simplification from \Cref{rmk:E1-explicit}
  to the map $\wt{2}^{3} : \Ss^{0,3} \to \nu\Ss$ we a obtain sequence of inductively defined obstructions
  \[ \theta_{k} \in [ \Sigma^{-3, 3}( \Ss^{2,1} )^{\otimes k}    , \nu\Ss/\wt{2}^{3} ] \qquad \text{ for } k \geq 2\]
  whose vanishing implies $\nu\Ss/\wt{2}^3$ admits an $\E_1$-algebra structure.
  On the other hand, the vanishing line from \cite[Prop. 15.8]{boundaries}\footnote{For the interested reader we note that the part of this proposition we use is essentially equivalent to Adams' vanishing line in the cohomology of the Steenrod algebra from \cite{AdamsPer}.} \footnote{We warn the reader that our convention for the indexing of bigraded spheres (see convention (6)) differs from that used in loc. cit.} says that
  \[ [\Ss^{w,s}, \nu\Ss/\wt{2}^3] = 0\]
  when $s > \frac{1}{2} w + 3$.
  In particular, this implies that $\theta_k$ is zero because it lies in a zero group!
\end{proof}

\begin{sseqdata}[ name = ASS, xscale=0.8, yscale=0.8, x range = {0}{12}, y range = {0}{9}, x tick step = 2, y tick step = 2, axes type = frame, class labels = {left}, classes = fill, grid = crossword, Adams grading, lax degree]

  \class(0,0)
  \class(0,1) \structline
  \class(0,2) \structline
  \class(1,1) \structline(0,0)
  \class(2,2) \structline
  \class(3,1) \structline(0,0)
  \class(3,2) \structline \structline(0,1)
  \class(3,3) \structline \structline(0,2) \structline(2,2)

  \class(2,3)
  \class(3,4) \structline
  \class(4,3) 
  \class(4,4) \structline
  \class(4,5) \structline \structline(3,4)

  \class(6,2) \structline(3,1)
  \class(7,1)
  \class(7,2) \structline
  \class(7,3) \structline
  \class(8,2) \structline(7,1)
  \class(9,3) \structline \structline(6,2)
  
  \class(8,3)
  \class(9,4) \structline
  
  \class(10,7)
  \class(11,8) \structline
  \class(12,7) 
  \class(12,8) \structline
  \class(12,9) \structline \structline(11,8)
  
  \class(7,4) \structline(4,3)
  
  \class(9,4)
  \class(10,5) \structline \structline(7,4) 
  
  \class(8,4)
  \class(8,5) \structline
  \class(8,6) \structline
  
  \class(9,5) \structline(8,4)
  \class(10,6) \structline
  \class(11,5) \structline(8,4)
  \class(11,6) \structline \structline(8,5)
  \class(11,7) \structline \structline(10,6) \structline(8,6)

  \class(9,5)
  \class(10,6) \structline

  \class[red](-2,2) 
  \class[red](14,10) \structline[green!50!black](-2,2)

  \class[red, "\theta_2" below](1,5)
  \class[red, "\theta_3" below](3,6)
  \class[red, "\theta_4" below](5,7)
  \class[red, "\theta_5" below](7,8)
  \class[red, "\theta_6" below](9,9)
  \class[red, "\theta_7" below](11,10)
  \class[red, "\theta_8" below](13,11)

\end{sseqdata}

\begin{figure}
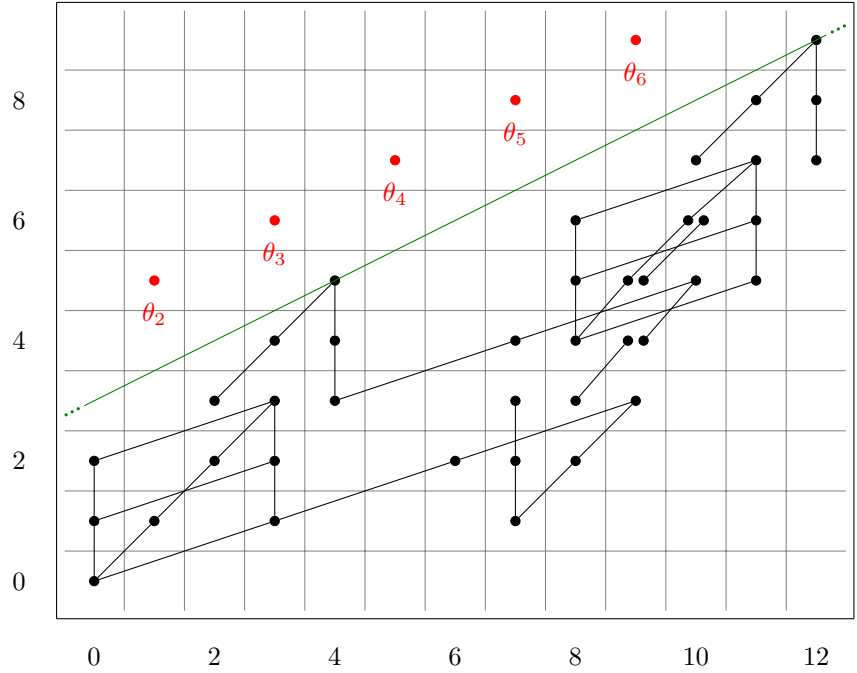

  \centering
  The $\F_2$-synthetic homotopy groups of $\nu\Ss/\wt{2}^3 $.\vspace{0.3cm}
  \printpage[ name = ASS, page = 2 ]  
  \caption{\footnotesize A picture of the $\F_2$-synthetic homotopy groups of $\nu\Ss/\wt{2}^3 $.    
    Black dots indicate non-$\tau$-torsion classes and we suppress all $\tau$-multiples.
    The green line is the vanishing line above which the bigraded homotopy groups are zero.
    The red classes are the obstructions $\theta_k$.}
  \label{fig:syn}
\end{figure}

Building on the proof of \Cref{thm:mod8} we prove the more complicated \Cref{thm:mod2n} in an almost identical way.

\begin{thm} \label{thm:mod2n}
  $\Ss/2^{q}$ admits the structure of an $\E_n$-algebra for every $q \geq \frac{3}{2}(n+1)$.
\end{thm}

\begin{proof}
  This time we will show that in the $\F_2$-synthetic category
  $\Ss/\wt{2}^{q}$ admits an $\E_n$-algebra structure.
  Applying the obstruction theory from \Cref{prop:obs} 
  to the map $\wt{2}^{q}$ we a obtain sequence of inductively defined obstructions
  \[ \theta_{k, \alpha} \in
    [ \Sigma^{-2-n-c_\alpha, 2+n+c_\alpha}( \Ss^{1+n,q-1-n} )^{\otimes k},\  \Ss/\wt{2}^{q} ] \]
  where $k \geq 2$ and $0 \leq c_\alpha \leq (n-1)(k-1)$ 
  whose vanishing implies $\Ss/\wt{2}^{q}$ admits an $\E_n$-algebra structure.
  Comparing the bidegree of $\theta_{k,\alpha}$ to the vanishing region from \Cref{lem:small-van} below 
  we again find that $\theta_{k,\alpha}$ lies in a zero group since
  \[ \left( (q-1-n)k + 2 + n + c_\alpha \right) > \frac{1}{2}\left( (n+1)k - 2 - n - c_\alpha \right) + q. \]
\end{proof}

\begin{lem} \label{lem:small-van} 
  If $ s > \frac{1}{2}w + q$, then $\pi_{w,s}(\nu\Ss/\wt{2}^q) = 0$.
\end{lem}

\begin{proof}
  In this lemma we invoke the machinery of vanishing lines from \cite[\S11]{boundaries} and
  our proof is patterned on the proof of \cite[Prop. 15.8]{boundaries}.
  
  We proceed by induction with $q=1$ as our base case (which is already covered by loc. cit.).
  For the inductive step we apply \cite[Lem. 11.11]{boundaries} to the cofiber sequence
  \[ \Sigma^{0,1}\nu\Ss/\wt{2}^{q-1} \to \nu\Ss/\wt{2}^q \to \nu\Ss/\wt{2}.\]
\end{proof}

\section{Deforming $\CC$} \label{sec:Bockstein}

% Recall that we have fixed a stable presentably $\E_n$-monoidal category $\CC$
% and a cofiber sequence 
% \[ I \xrightarrow{v} \o_\CC \to \o_\CC/v \]
% such that $\o_\CC/v$ admits a right unital multiplication.
In this section we use the machinery of \cite{PP21} to construct a deformation of $\CC$ in which we can run the arguments of the previous section in order to prove \Cref{thm:En-quo}.

\begin{rec} \label{rec:def}
  Given an epimorphism class $Q$ on a stable category $\mathcal{E}$ such that $\mathcal{E}$ has enough $Q$-injectives
  Patchkoria and Pstr\k{a}gowski define a prestable category
  $\mathcal{D}_{\geq 0}^\omega(\mathcal{E};\,Q)$
  with associated stable category $\mathcal{D}^\omega(\mathcal{E};\,Q)$
  which fits into a diagram
  \begin{center}
    \begin{tikzcd}
      & \mathcal{D}_{\geq 0}^\omega(\mathcal{E};\,Q) \ar[dr, "(-)^{\tau=1}"] & \\
      \mathcal{E} \ar[ur, "\nu"] \ar[rr, "\Id"] & & \mathcal{E}
    \end{tikzcd}
  \end{center}  
  such that
  \begin{enumerate}
  \item $\mathcal{D}_{\geq0}^\omega(\mathcal{E};\,Q)$ has finite limits and colimits.    
  \item $\nu$ is fully faithful.
  \item $\mathcal{D}_{\geq0}^\omega(\mathcal{E};\,Q)$ is generated under finite colimits by the image of $\nu$.
  \item $\nu$ preserves those cofiber sequences $ a \to b \to c $ for which $b \to c$ is $Q$-epi.
  \item $\mathcal{D}_{\geq 0}^\omega(\mathcal{E};\,Q)$ is equipped with an automorphism $[1]$ and an equivalence
    \[ \nu(-)[1] \simeq \nu(\Sigma -). \]
  \item $(-)^{\tau=1}$ is the localization of $\mathcal{D}^\omega(\mathcal{E};\,Q)$ at (integer suspensions of) the assembly maps
    \[ \tau_X : \Sigma \nu X [-1] \to \nu X. \]
  \item If $\II$ is $Q$-injective and $X \in \mathcal{E}$, then $[\Sigma^{-s}\nu X,\ \nu \II] = 0$ for $s > 0$.
  \end{enumerate}
  For these claims see \cite[\S5]{PP21}, specifically 5.32, 5.34, 5.37, 5.47 and 5.60.
  \tqed  
\end{rec}

For the proof of our main theorem we will also need an $\E_m$-monoidal structure on our deformation of $\CC$. Although \cite{PP21} only considers monoidal structures in the case of $\Sp$, as it turns out the general case is no more difficult and our treatment follows \cite[\S 5.5]{PP21} closely.

\begin{dfn}
  We say that an epimorphism class $Q$
  on a stably $\E_m$-monoidal category $\mathcal{E}$ (with $m \geq 2$) is \emph{$\otimes$-compatible}
  if for every $Q$-epi map $X \to Y$ and $Z \in \mathcal{E}$ the map
  $X \otimes Z \to Y \otimes Z$ is $Q$-epi as well.\footnote{Note that because $m \geq 2$ it doesn't matter whether $Z$ is on the left or right in this definition. The potential discrepancy between left $\otimes$-compatible and right $\otimes$-compatible in the $\E_1$-monoidal case is the root of our restriction to $m \geq 2$.}
  \tqed
\end{dfn}

\begin{prop}
  If $Q$ is a $\otimes$-compatible epimorphism class on
  a stably $\E_m$-monoidal category $\mathcal{E}$ with $m \geq 2$,
  then $\mathcal{D}^\omega(\mathcal{E}; Q)$ admits an exact $\E_m$-monoidal structure,
  compatible with the prestable structure, such that $\nu$ and $(-)^{\tau=1}$ are $\E_m$-monoidal.
  % \begin{enumerate}
  % \item $\nu$ is $\E_n$-monoidal (and in particular $\nu \o$ is the unit),
  % \item $\tau$ is compatible with the tensor product in the sense that
  %   $\tau_X \simeq \tau_\o \otimes \nu X$,
  % \item the cofiber of $\tau^k$ at the unit, $C\tau^k$, is an $\E_n$-algebra,
  % \item every bounded object is a $C\tau^k$-module for $k \gg 0$,
  % \end{enumerate}      
\end{prop}

\begin{proof}
  The construction of the $\E_m$-monoidal structure, compatibility with the prestable structure
  and the $\E_m$-monoidality of $\nu$ are proved in the same way as in \cite[Prop. 5.69]{PP21}.
  Using the fully-faithfulness of $\nu$ we can identify $\tau_X$ with $\tau_\o \otimes \nu X$
  (and consequently we drop the subscript from $\tau_\o$ going forward).
  The $\E_m$-monoidality of $(-)^{\tau=1}$ now follows from describing this localization as inverting the map $\tau$ in a monoidal way.
  % Using the fact that the cofiber of $\tau$ at $X$ lies in the heart
  % we can rephrase \Cref{rec:def}(4) as saying that $(-)^{\tau=1}$ is the quotient by the bounded objects.
  % In order to give $(-)^{\tau=1}$ an $\E_n$-monoidal structure we just need to check that if $X$ is bounded, then $X \otimes Y$ is bounded. Using exactness of $-\otimes-$ and the fact that the image of $\nu$ generates $\mathcal{D}^\omega(\CC^\omega;\,Q)$ under finite (co)limits it suffices to handle the case where $Y \simeq \nu A$. Since $X$ is bounded the map
  % \[ \tau^k \otimes X : (\Sigma^k \nu \o [-k]) \otimes X \to X \]
  % is null for $k \gg 0$.
  % Thus, it suffices to argue that $C\tau^k \otimes X \otimes \nu A$ is bounded.
  % Using exactness again we can replace $X$ by some $\nu B$.
  % Since $\nu$ is $\E_n$-monoidal we are now reduced to noting that cofiber of
  % \[ \tau^k : \Sigma^{k} \nu Z [-k] \to \nu Z \]
  % is bounded.
\end{proof}

We are now ready to introduce the specific deformation of interest to us.

\begin{dfn}
  Let $\mathcal{Q}$ denote the epimorphism class of maps $X \to Y \in \CC^\omega$ which are split epi upon tensoring with $\o_\CC/v$.\footnote{That these maps form an epimorphism class follows from \cite[Examples 3.4 and 3.6]{PP21}.}
  \tqed
\end{dfn}

\begin{lem} \label{lem:enough-inj}
  The epimorphism class $\mathcal{Q}$ enjoys the following properties:
  \begin{enumerate}
  \item Every object of the form $\o_\CC/v \otimes X$ is $\mathcal{Q}$-injective.
  \item The map $X \to \o_\CC/v \otimes X$ is $\mathcal{Q}$-mono.
  \item $\CC^\omega$ has enough $\mathcal{Q}$-injectives.
  \item $\mathcal{Q}$ is $\otimes$-compatible.
  \end{enumerate}
\end{lem}

\begin{proof}
  Using the right unital multiplication on $\o_\CC/v$ 
  conclusions (1) and (2) follow from the same argument as in the proof of \cite[Lem. 5.67]{PP21}.
  (3) follows from (1) and (2).
  Again using the right unital multiplication on $\o_\CC/v$,
  conclusion (4) follows from the same argument as in \cite[Lem. 5.68]{PP21}.%\footnote{This is essentially the only point where the condition $m \geq 2$ is important. It would be interesting to know what additional conditions on $\o_\CC/v$ are necessary for the $(\o_\CC/v)$-split epi deformation to be monoidal when $m=1$.}
\end{proof}

There is one more modification we need to make in order to link up with the material from \Cref{sec:budget}: We want to have a \emph{presentable} category deforming $\CC$.

\begin{cnstr} \label{cnstr:renorm2}
  Recall that we have already arranged in \Cref{rmk:force-pres} that $\o$ and $\II$ are compact and
  $\CC$ is generated under tensor products and finite (co)limits by these generators.
  We let $\Def(\CC;\, \mathcal{Q})$ denote the ind-completion of $\mathcal{D}^\omega(\CC^\omega; \mathcal{Q})$.
  This renormalization fits into a diagram of presentably $\E_m$-monoidal categories and $\E_m$-monoidal, filtered colimit preserving functors
  \begin{center}
    \begin{tikzcd}
      & \Def(\CC;\,\mathcal{Q}) \ar[dr, "(-)^{\tau=1}"] & \\
      \CC \ar[ur, "\nu"] \ar[rr, "\Id"] & & \CC
    \end{tikzcd}
  \end{center}
  which agrees with the one from \ref{rec:def} on compact objects.
  \tqed
\end{cnstr}

We end the section by proving a vanishing lemma which serves as the analog of \Cref{lem:small-van} in $\Def(\CC;\,\mathcal{Q})$.

\begin{cnstr} \label{cnstr:v-tilde}
  Since $\o_\CC \to \o_\CC/v$ is $(\o_\CC/v)$-split mono we have a cofiber sequence
  \[ \nu\o_\CC \to \nu (\o_\CC/v) \to \nu (\Sigma \II). \]
  We write $\wt{v}$ for the associated boundary map
  \[ \wt{v} : \Sigma^{-1} \nu \II [1] \to \nu\o_\CC. \]
  \tqed
\end{cnstr}

\begin{lem} \label{lem:vanishing}
  The cofiber of the map $\wt{v}^q : (\Sigma^{-1} \nu \II [1])^{\otimes q} \to \nu \o_\CC $
  in $\Def(\CC;\, \mathcal{Q})$ has
  \[ [\Sigma^{-s}\nu X,\ \nu\o_\CC/\wt{v}^q ] = 0 \]
  for every $X \in \CC^\omega$ and $s \geq q$.
\end{lem}

\begin{proof}
  The restriction to $X \in \CC^\omega$ lets us move back to $\mathcal{D}^\omega(\CC^\omega;\,\mathcal{Q})$.
  We proceed by induction on $q$ using the cofiber sequence
  \[ (\Sigma^{-1}\nu \II[1])^{\otimes q-1} \otimes \nu\o_\CC/\wt{v} \to \nu\o_\CC/\wt{v}^{q} \to \nu\o_\CC/\wt{v}^{q-1}. \]
  Using the equivalence between $\nu(\o_\CC/v)$ and $\nu\o_\CC/\wt{v}$ from \Cref{cnstr:v-tilde}
  and the fact that $\nu$ is monoidal we have an equivalence
  \[ (\nu \II[1])^{\otimes q-1} \otimes \nu\o_\CC/\wt{v} \simeq \nu(\Sigma^{q-1} \II^{\otimes q-1} \otimes \o_\CC/v). \]
  Since $\Sigma^{q-1} \II^{\otimes q-1} \otimes \o_\CC/v$ is $\mathcal{Q}$-injective (see \Cref{lem:enough-inj}) we then have that 
  \[ [\Sigma^{-s}\nu X,\ \ (\Sigma^{-1}\nu \II[1])^{\otimes q-1} \otimes \nu\o_\CC/\wt{v}] = 0 \]
  (see \ref{rec:def}(7))
  for every $X \in \CC^\omega$ and $s \geq q$.
\end{proof}

\section{$\E_n$-algebra quotients} \label{sec:quo}

With all the preparation finished we are now ready to prove the main theorem.
Our proof follows the pattern established in \Cref{sec:even-Moore} with the key difference being that we use the stable, presentably $\E_n$-monoidal deformation $\Def(\CC; \mathcal{Q})$ of $\CC$ from the previous section in place of $\Syn_{\F_2}$.

\begin{dfn}\label{dfn:v-compatible}
  We say that an $\E_n$-algebra structure on $\o_\CC/v^q$ is \emph{$v$-compatible} if it equivalent to
  $ (\nu\o_\CC/\wt{v}^q)^{\tau=1}$ for some $\E_n$-algebra structure on the quotient
  $\nu\o_\CC/\wt{v}^q$ in $\Def(\CC; \mathcal{Q})$.
  \tqed
\end{dfn}

\begin{thm}[\Cref{thm:En-quo}] \label{thm:En-quo-body}
  If $\o_\CC/v$ admits a right unital multiplication,
  then there exists a unique up to equivalence $v$-compatible $\E_n$-algebra $\o/v^{q}$ for each $q > n$
  and these $\E_n$-algebras fit into a tower
  \[ \cdots \longrightarrow \o/v^{n+3} \longrightarrow \o/v^{n+2} \longrightarrow \o/v^{n+1}. \]
\end{thm}

\begin{proof}
  % Using the right unital multiplication on $\o_\CC/v$ we constructed a stable, presentably $\E_n$-monoidal deformation $\Def(\CC; \mathcal{Q})$ of $\CC$ in the previous section.
  % % \Cref{lem:En-on-Bockstein} then gives us an $\E_n$-monoidal structure on $\Def(\CC; \o_\CC/v)$
  % % such that $\nu$ and $\tau^{-1}$ are $\E_n$-monoidal.  
  % We will prove the theorem by showing that $\nu\o_\CC/\wt{v}^{q}$ admits an $\E_n$-algebra structure,
  % with the $\E_n$-algebra structure on $\o_\CC/v^{q}$ then arising from the equivalence
  % $(\nu\o_\CC/\wt{v}^{q})[\tau^{-1}] \simeq \o_\CC/v^{q} $.

  We begin by producing an $\E_n$-algebra structure on $\nu\o_\CC/\wt{v}^q$ for $q > n$.
  Applying the obstruction theory from \Cref{prop:obs} and \Cref{cor:obs-degs}
  to the map $\wt{v}^{q}$ we a obtain sequence of inductively defined obstructions
  \[ \theta_{k,\alpha} \in
    [ \Sigma^{-2-n-c_\alpha}( \Sigma^{n+1} (\Sigma^{-1}\nu \II [1])^{\otimes q} )^{\otimes k} ,\ \  \nu\o_\CC/\wt{v}^{q} ] \]
  where $k \geq 2$ and  $0 \leq c_\alpha \leq (n-1)(k-1)$
  whose vanishing implies $\nu\o_\CC/\wt{v}^{q}$ admits an $\E_n$-algebra structure.
  Since $\nu$ is $\E_n$-monoidal we can rewrite the source of $\theta_{k,\alpha}$ as 
  \[ \Sigma^{(n+1-q)k-n-2-c_\alpha} \nu \left( (\Sigma \II)^{\otimes qk} \right). \]
  Since $q \geq n+1$ and $n+2+c_\alpha \geq n+1$, \Cref{lem:vanishing} now tells us that the group in which $\theta_{k,\alpha}$ lives is trivial.

  In order to prove the uniqueness statement and produce the desired tower
  we now examine the space of maps from the $\E_n$-algebra $\nu\o_\CC/\wt{v}^q$ constructed above
  to any other $\E_n$-algebra $R$ with the underlying object $\nu\o_\CC/\wt{v}^w$.
  The pushout squares from \Cref{prop:obs} can be interpreted as providing an obstruction theory for producing $\E_n$-algebra maps $\nu\o_\CC/\wt{v}^q \to R$.
  Using \Cref{cor:obs-degs} we can sub-divide these obstructions into obstructions
  \[ \gamma_{k,\alpha} \in
    [ \Sigma^{(n+1-q)k-n-1-c_\alpha} \nu \left( (\Sigma \II)^{\otimes qk} \right) ,\ \  \nu\o_\CC/\wt{v}^{w} ] \]
  where $k \geq 1$ and $0 \leq c_\alpha \leq (n-1)(k-1)$. 
  When $q \geq w$ \Cref{lem:vanishing} tells us that these groups are trivial and therefore we obtain the desired map.
  In the case $q=w$ the obstruction $\gamma_1$ is the composite
  \[ \Sigma^{-q}\nu((\Sigma \II)^{\otimes q} ) \xrightarrow{\wt{v}^q} \nu\o_\CC \to \nu\o_\CC/\wt{v}^{q} \]
  and if we pick the nullhomotopy of $\gamma_1$ to be one which makes this into a cofiber sequence, then we will obtain an equivalence of $\E_n$-algebras $\nu\o_\CC/\wt{v}^q \to R$, proving the uniqueness assertion.
\end{proof}

% \begin{rmk}
%   In fact, since \Cref{lem:vanishing} gives us more degrees of vanishing than was used in the proof of \Cref{thm:En-quo-body} we can conclude that the space of $\E_n$-algebra structures on $\o_\CC/v^{n+1}$ produced by the constructions in this proof is connected.
%   In plain term we are saying that \Cref{thm:En-quo-body} constructs a single, specific $\E_n$-algebra structure on $\o_\CC/v^{n+1}$.
%   \tqed
% \end{rmk}

We conclude by deducing the remaining theorems from the introduction % \ref{thm:mult-Moore}, \ref{thm:gen-Moore} and \ref{thm:ring-quo} from the introduction
as corollaries of the main theorem.

\begin{cor}[\Cref{thm:ring-quo}]
  Suppose we are given an $\E_{m+1}$-algebra $A \in \Sp$ with $m \geq 2$
  and a class $v \in \pi_{*}(A)$
  such that $A/v$ admits a unital multiplication.
  Then $A/v^q$ admits an $\E_n$-$A$-algebra structure for each $n \leq m$ and $q>n$.
\end{cor}

\begin{proof}
  Since $A$ is an $\E_{m+1}$-algebra it has an $\E_m$-monoidal category of left modules in which we can apply \Cref{thm:En-quo} to $v$.
\end{proof}

Note that the statement proved here is slightly different from the one which appeared in the introduction. In order to bridge the gap we include the next lemma, which is likely well known to experts, but for which we could not find a reference in the literature.

\begin{lem}
  Given an $\E_{m+1}$-algebra $A \in \Sp$ with $m \geq 2$ and a class $v \in \pi_{2w}(A)$
  we can identify the obstruction $\overline{Q_1}(v)$ to $A/v$ admitting a unital multiplication with the reduction mod $v$ of $Q_1(v)$.
\end{lem}

\begin{proof}
  In proving this lemma we pick up where \Cref{lem:A2} left off, passing to the filtered setting.
  We write $X[\tau]$ for the image of $X \in \Sp$ under the unit map $\Sp \to \Sp^{\Fil}$.\footnote{This notation is meant to evoke that the underlying graded object of $X[\tau]$ looks like a free module over $\Ss[\tau]$ on a copy of $X$ placed in degree $0$.}

  The $\E_1$-cofiber $A[\tau]/\!\!/\tau v$ admits a second description as the relative tensor product of $A[\tau]$ with $\Ss[\tau]$ over a free filtered $\E_2$-algebra on a class $x$ in degree $2w$ and filtration $1$ which maps to $\tau v$ in $A[\tau]$ and to zero in $\Ss[\tau]$.
  The arity $2$ component of the free $\E_2$-algebra on $\Ss^{2w}$ is given by
  $\Ss^{4w} \oplus \Ss^{4w+1}$ where the bottom cell is $x^2$ and the top cell is $Q_1(x)$. 
  Thus, through filtration $2$, we can replace the free $\E_2$-algebra with a free $\E_1$-algebra on two classes $x$ and $Q_1(x)$ such that $x$ maps to $\tau v$ and $Q_1(x)$ maps to $\tau^2 Q_1(v)$.
  This lets us identify the filtration $2$ component of $A[\tau]/\!\!/\tau v$ with
  \[ \cof \left( \Sigma^{2w}A \oplus \Sigma^{4w+1}A \xrightarrow{(v,\ Q_1(v))} A \right). \]
  Unrolling the definition we see that the attaching map of the top cell to the copy of $A/v$ is the obstruction $\overline{Q_1}(v)$ of \Cref{lem:A2} and we can thereby identify $\overline{Q_1}(v)$ with the reduction mod $v$ of $Q_1(v)$.    
\end{proof}

\begin{rmk} \label{rmk:sq-fine}
  The Cartan formula from \cite[Prop. V.1.10]{Hinf}\footnote{Here we are using that we are in the stable range where $\E_\infty$ and $\E_{m+1}$ power operations agree.} tells us that for an $\E_{m+1}$-algebra $A$ with $m \geq 2$ and $x,y \in \pi_{2*}(A)$ we have
  \[ Q_1(xy) \equiv Q_1(x)y^2 + x^2Q_1(y) + c \eta x^2y^2 \]
  for some integer $c$.
  In particular we find that
  \[ \overline{Q_1}(x^2) \equiv Q_1(x^2) \equiv 2x^2Q_1(x) + c\eta x^4 \equiv 0 \pmod{x^2}. \]
  Thus, the square of an even dimensional class always satisfies the conditions of \Cref{thm:ring-quo}.
  \tqed
\end{rmk}

\begin{cor}[\Cref{thm:mult-Moore}, odd primes]
  Applying \Cref{thm:ring-quo} with $\CC=\Sp$ and $v = p$ we obtain
  an $\E_n$-algebra structure on $\Ss/p^{n+1}$.
\end{cor}

\begin{rmk} 
  Although $\Ss/2$ does not admit a unital multiplication,
  $\overline{Q_1}(4) = 0$ and therefore $\Ss/4$ admits a unital multiplication.
  If we apply \Cref{thm:ring-quo} with $v = 4$, then we obtain an $\E_n$-algebra structure on $\Ss/2^{2(n+1)}$.
  Note that this is \emph{less} structure than is provided by \Cref{thm:mod2n}.
  This discrepancy suggests that there is still room for improvement at the prime $2$.
  Specifically, we suspect that the optimal value of $q$ for which $\Ss/2^q$ is an $\E_n$-algebra is not much larger than $n$ in general.
  \tqed
\end{rmk}

% \begin{proof}
%   All we need to check is that $\Ss/p$ admits a right unital multiplication, which is standard.
% \end{proof}

\begin{cor}[\Cref{thm:gen-Moore}]
  For each $h$ and $n$ there exists a generalized Moore spectrum $\Ss/(p^{i_0},\dots,v_{h-1}^{i_{h-1}})$ of type $h$ which admits an $\E_n$-algebra structure.
\end{cor}

\begin{proof}
  We proceed by induction on $h$.
  Suppose, by induction, that we have an $\E_{n+1}$-algebra structure on a type $h-1$ generalized Moore spectrum $M$.
  The periodicity theorem of \cite{NilpII} guarantees we can find a $v_{h-1}$-self map $v \in \pi_*M$.
  To conclude we apply \Cref{thm:ring-quo} to $v^2$ (see \Cref{rmk:sq-fine}).
  \qedhere
  
  % At odd primes this works as stated.
  % At $p=2$ we make the following modification:
  % replace $v$ by $v^{2^N}$ for $N \gg 0$ then using the derivation-type relation for $\overline{Q_1}(-)$ from \cite{} we have $\overline{Q_1}(v^{2^N}) = 2^N v^{2(2^{N}-1)} \overline{Q_1}(v) = 0$.
  % therefore we can replace $v$ by $v^{2^N}$ for $N \gg 0$ in order to guarantee the vanishing of $Q_1$ needed to apply \Cref{thm:En-quo}.  
  
  % For this we need that the cofiber of $v$ admits a right unital multiplication.    
  
  %\Cref{lem:A2} together with the fact that $Q_1(-)$ vanishes on $\E_3$-algebras on which $2$ is invetible to conclude.
  % (and therefore the Euler characteristic condition is satisfied).
\end{proof}

%% file: barcobar.tex
\section{Bar-cobar duality for graded $\E_n$-algebras}
\label{sec:barcobar}

In this appendix we show that Lurie's bar-cobar duality for $\E_n$-algebras in an $\E_n$-monoidal category can be upgraded to an equivalence in the positively graded setting.
Our proof follows \cite[\S4]{ChiralKoszul} closely enough that it is worth pointing out why a simple citation is insufficient.
\begin{enumerate}
\item[(a)] In \cite{ChiralKoszul} the underlying category $\EE$ is \emph{symmetric monoidal}.
\item[(b)] In \cite{ChiralKoszul} Koszul duality has target divided power coalgebras over the Koszul dual operad.
\end{enumerate}
By contrast, Lurie's (iterated) bar-cobar duality has the advantage that it is defined for $\E_n$-monoidal categories, but the disadvantage that it is not immediate that this duality is the same one one would expect to obtain using operadic Koszul duality (together with the Koszul self-duality of the $\E_n$-operad).

\begin{rmk}
  In the long-run the author would like to see this appendix supplanted by an extension of operadic Koszul duality to algebras in categories which are not symmetric monoidal.
  \tqed
\end{rmk}

\begin{cnv}
  Throughout this appendix $\EE$ will denote a stable, presentably $\E_n$-monoidal category.
  \tqed
\end{cnv}

As in the body of the paper 
$\EE^{\Gr}$ denotes the category of graded objects in $\EE$
and this category is also stable and presentably $\E_n$-monoidal.
Crucially for this appendix, in $\EE^{\Gr}$ limits and colimits are computed component-wise.
In particular we have the following lemma:

\begin{lem} \label{lem:sum-prod}
  Given a collection of objects $\{X_\alpha\}_{\alpha \in A}$ in $\EE^\Gr$
  such that only finitely many of the $X_\alpha$ are non-zero in each degree,  
  the natural map
  \[ \textstyle\bigoplus_A X_\alpha \to \prod_A X_\alpha \]
  is an equivalence.
\end{lem}

We will also need several variants of $\EE^\Gr$

\begin{dfn}
  \ 
  \begin{itemize}
  \item $\EE_{\geq 0}^\Gr$ is the full subcategory of $\EE^\Gr$ on objects which vanish in negative degrees.
  \item $\EE_+^\Gr$ is the full subcategory of $\EE^\Gr_{\o/-/\o}$ on the objects which vanish 
    in negative degrees and are equivalent to the unit in degree $0$.
    We say that the objects of $\EE_+^\Gr$ are \emph{positively graded}.
  \item Given an object $X \in \EE^\Gr_{\o/-/\o}$ we write $\overline{X}$ for the object obtained by splitting off the copy of the unit.\footnote{The functor which sends $X \in \EE_+^\Gr$ to $\overline{X}$ gives an equivalence between $\EE_+^\Gr$ and the category of graded objects concentrated in positve degrees. Note however, that this equivalence is not monoidal.}
  \item We say an object in $\EE^\Gr$ is \emph{thin} if it vanishes in all but finitely many degrees.
    We write $(\EE^\Gr)^{\mathrm{thin}}$, $(\EE_{\geq 0}^\Gr)^{\mathrm{thin}}$ and $(\EE_{+}^\Gr)^{\mathrm{thin}}$
    for the respective subcategories of thin objects.    
  \end{itemize}
  \tqed
\end{dfn}

\begin{rec}
  In \cite[\S 5.2.3]{HA} Lurie constructs a bar-cobar adjunction
  \[ \Barn : \Alg_{\E_n}^{\aug}(\EE^{\Gr}) \rightleftharpoons \mathrm{coAlg}_{\E_n}^{\aug}(\EE^{\Gr}) : \Cobarn \]  
  (see \cite[5.2.3.6 and 5.2.3.9]{HA} specifically).
  Since $\EE_+^\Gr$ is a full subcategory of $\EE^\Gr_{\o/-/\o}$ closed under tensor products, limits and colimits, the bar-cobar adjunction restriction to an adjunction between these subcategories
  \[ \Barn : \Alg_{\E_n}(\EE_+^{\Gr}) \rightleftharpoons \cAlg_{\E_n}(\EE_+^{\Gr}) : \Cobarn \]
  (see \cite[5.2.3.11]{HA}).
  In the case $n=1$, $\Bar$ (resp. $\Cobar$) is computed by a bar (cobar) construction
  \cite[5.2.2.17]{HA}. 
  \tqed
\end{rec}

% \begin{rmk}
%   The functor which sends an object $X \in \EE_+^\Gr$ to $\overline{X}$ gives an equivalence
%   between $\EE_+^\Gr$ and the category of graded objects concentrated in positve degrees.
%   Thus, $\EE_+^\Gr$ has all limits and colimits. Note however, that this equivalence is not monoidal.
%   \tqed
% \end{rmk}

The main theorem of this appendix is that in the positively graded setting we can upgrade Lurie's bar-cobar adjunction to an equivalence.

\begin{thm} \label{thm:En-bar-cobar}
  The bar-cobar adjunction
  \[ \Barn : \Alg_{\E_n}(\EE_+^{\Gr}) \rightleftharpoons \cAlg_{\E_n}(\EE_+^{\Gr}) : \Cobarn \]
  is an equivalence.  
\end{thm}

\begin{rmk}
  Similar ideas were considered in \cite{KrauseApp} where, following \cite{ChiralKoszul}, Krause proves the $n=1$, $\EE = \Mod(\Z)$ case of \Cref{thm:En-bar-cobar}.
  \tqed
\end{rmk}

% By restricting a presentable category $\DD$ to its subcategory of $\kappa$-compact objects for sufficiently large $\kappa$, applying \Cref{thm:En-bar-cobar} and then passing back to $\DD$ via ind-completion we obtain the following corollary.

% \begin{cor}
%   If $\DD$ is stable and presentably $\E_n$-monoidal, then
%   the bar-cobar adjunction
%   \[ \Barn : \Alg_{\E_n}(\DD_+^{\Gr}) \rightleftharpoons \cAlg_{\E_n}(\DD_+^{\Gr}) : \Cobarn \]
%   is an equivalence.  
% \end{cor}

\subsection{The proof of \Cref{thm:En-bar-cobar}}
\label{subsec:bar-cobar-pf}\ 

The proof of \Cref{thm:En-bar-cobar} will proceed by induction on $n$ 
and the bulk of the work lies in handling the base-case $n=1$.
The preparation for this proof will occupy us for the next couple pages.

\begin{cnstr} \label{cnstr:partial-tor-res}
  Given a positively graded $\E_1$-algebra $A$ in $\EE$
  we can consider the bar filtration on the underlying object of $\Bar(A)$,
  \begin{center}
    \begin{tikzcd}
      \displaystyle\colim_{\Delta_{\leq 0}^\op} A^{\otimes \bullet} \ar[d, "\simeq"] \ar[r] &
      \displaystyle\colim_{\Delta_{\leq 1}^\op} A^{\otimes \bullet} \ar[d] \ar[r] &
      \displaystyle\colim_{\Delta_{\leq 2}^\op} A^{\otimes \bullet} \ar[d] \ar[r] &
      \cdots \ar[r] &
      \Bar(A) \\
      \o &
      \Sigma \overline{A} &
      (\Sigma \overline{A})^{\otimes 2}       
    \end{tikzcd} 
  \end{center}
  which is an $\omega$-indexed filtration of the underlying object of $\Bar(A)$ with associated graded given by $(\Sigma \overline{A})^{\otimes k}$.

  Dually, given a positively graded $\E_1$-coalgebra $C$ in $\EE$
  we can consider the cobar filtration of $\Cobar(C)$,
  which is an $\omega$-indexed tower with limit the underlying object of $\Cobar(C)$ and
  associated graded given by $(\Sigma^{-1} \overline{C})^{\otimes k}$.
  \tqed
\end{cnstr}

\begin{rmk} \label{rmk:fin-stabilize}
  The key observation in this appendix is that 
  $(\Sigma^{-1} \overline{C})^{\otimes k}$ is concentrated in degrees $\geq k$,
  therefore in any fixed degree $\Cobar(C)$ is computed by a \emph{finite} limit.
  \tqed
\end{rmk}

\begin{lem} \label{lem:cobar-sifted}
  The functor $\Cobar : \cAlg(\EE) \to \Alg(\EE)$ commutes with sifted colimits.
\end{lem}

\begin{proof}  
  Since the underlying object functor $\Alg(\EE) \to \EE$
  is conservative and commutes with sifted colimits it will suffice to
  show that the composite of $\Cobar$ with this functor commutes with sifted colimits.  
  % By \Cref{lem:iterated}, $\Cobarn$ is computed by an iterated cobar construction, so it suffices to prove the lemma in the case $n=1$.

  Suppose we have a sifted diagram $F : D \to \cAlg(\EE)$.
  Using that fact that $\Cobar$ is computed by a cobar construction
  we have equivalences
  \begin{align*}
    \left(\colim_{d \in D} \Cobar( F(d) ) \right)_k    
    &\simeq \colim_{d \in D} \varprojlim_s \left( \lim_{\Delta_{\leq s}} F(d)^{\otimes \bullet} \right)_k 
    \simeq \varprojlim_s \colim_{d \in D} \left( \lim_{\Delta_{\leq s}} F(d)^{\otimes \bullet} \right)_k \\
    &\simeq  \varprojlim_s \lim_{\Delta_{\leq s}} \left( \colim_{d \in D} F(d)^{\otimes \bullet} \right)_k
  \end{align*}
  where the key step is using the fact that the cobar filtration stabilizes in finitely many steps
  (see \Cref{rmk:fin-stabilize}) to commute the colimit and infinite limit.
  Using the assumption that $D$ is sifted and the fact that the tensor product on $\EE^\Gr$ commutes with colimits seperately in each variable we have
  \[ \colim_{d \in D} F(d)^{\otimes s} \simeq \colim_{(d_1,\dots,d_s) \in D^{\times s}} F(d_1) \otimes \cdots \otimes F(d_s) \simeq \left( \colim_{d \in D} F(d) \right)^{\otimes s}. \]
  Feeding this into the previous equivalence we obtain the desired equivalence
  \[ \left(\colim_{d \in D} \Cobar( F(d) ) \right)_k \simeq \varprojlim_s \lim_{\Delta_{\leq s}} \left( \left( \colim_{d \in D} F(d) \right)^{\otimes \bullet} \right)_k \simeq \Cobar \left( \colim_{d \in D} F(d) \right)_k.\]
\end{proof}

\begin{lem} \label{cor:agree-through-k}
  A map of postively graded $\E_1$-algebras $A \to B$ in $\EE$ is an equivalence through degree $k$
  iff the map of postively graded $\E_n$-coalgebras $\Bar(A) \to \Bar(B)$ is an equivalence through degree $k$.

  Dually, a map of postively graded $\E_1$-coalgebras $C \to D$ in $\EE$ is an equivalence through degree $k$ iff the map of postively graded $\E_n$-algebras $\Cobar(C) \to \Cobar(D)$ is an equivalence through degree $k$.
  
  In particular this implies that both $\Bar$ and $\Cobar$ are conservative.  
\end{lem}

\begin{proof}
  Suppose the map $A \to B$ is an equivalence through degree $k$, but is not an equivalence in degree $k+1$. Examining the bar filtration we see that the cofiber of the map $A \to B$ has a filtration with associated graded given by
  \[ X_{s,j} \coloneqq \cof \left( (\Sigma^s \overline{A}^{\otimes s})_j \to (\Sigma^s \overline{B}^{\otimes s})_j \right) \]
  in degree $j$.
  Using the fact that $\overline{A}$ and $\overline{B}$ are concentrated in degrees $\geq 1$ and the map between them is an equivalence in degrees $\leq k$ we can read off that
  $X_{s,j} = 0$ for $j \leq s+k-1$ and $X_{1,k+1} \simeq \cof(\Sigma \overline{A}_{k+1} \to \Sigma \overline{B}_{k+1})$.
  In particular, this implies that
  \begin{enumerate}
  \item[(a)] the map $\Bar(A) \to \Bar(B)$ is an equivalence through degee $k$ and
  \item[(b)] in degree $k+1$ we have
    \[ \cof(\Bar(A) \to \Bar(B))_{k+1} \simeq \cof(\Sigma \overline{A}_{k+1} \to \Sigma \overline{B}_{k+1}) \not\simeq 0. \]
  \end{enumerate}

  The argument for $\Cobar$ is dual to the argument for $\Bar$.
\end{proof}

\begin{cnstr} \label{cnstr:pro-thin-embed}
  After passing to categories of large pro-objects the natural inclusion
  $(\EE^{\Gr})^{\mathrm{thin}} \to \EE^\Gr$ picks up a left adjoint and 
  we write $e$ for the composite
  \[ e: \EE^\Gr \to \Pro(\EE^\Gr) \to \Pro((\EE^\Gr)^{\mathrm{thin}}) \]
  of this left adjoint with the Yoneda embedding.
  Similarly, we have functors $e_{\geq 0}$ and $e_+$ in the positively graded setting. %\footnote{The reason we restricted to small $\EE$ is so that passing to pro-objects doesn't send us to a larger universe.}
  \tqed
\end{cnstr}

\begin{lem} \label{lem:pro-thin-is-pro-gr}
  There is a natural $\E_n$-monoidal equivalence
  \[ \Pro((\EE^\Gr)^{\mathrm{thin}}) \simeq \Pro(\EE)^{\Gr} \]
  which restricts to a similar equivalence in the positively graded setting.
  In particular this means every object in $\Pro((\EE^\Gr)^{\mathrm{thin}})$ is
  both the coproduct and the product of its components in each degree.
\end{lem}

\begin{proof}
  The key point in this lemma is that both of these catgories are the opposite of a huge presentable category with a large collection of compact objects. To prove the lemma it therefore suffices to argue that
  the full subcategory of compact objects in $\Ind(\EE^\op)^\Gr$ is equivalent to
  $((\EE^\Gr)^{\mathrm{thin}})^{\op}$.
  For this we observe that since
  $\Ind(\EE^\op)^\Gr \simeq \Sp^\Gr \otimes \Ind(\EE^\op)$
  the compact objects are generated under finite colimits by objects of the form $X(k)$ with $X \in \EE^\op$.

  The second conclusion follows from \Cref{lem:sum-prod} which applies since the category in question is now the opposite of a category of graded objects in a (huge) presentable category.
\end{proof}
 
\begin{lem} \label{lem:pro-thin-good}
  The functor $e_+$ of \Cref{cnstr:pro-thin-embed}  
  is fully faithful,
  $\E_n$-monoidal and
  colimit preserving.
\end{lem}

\begin{proof}
  The Yoneda embedding $\EE_+^\Gr \to \Pro(\EE_+^\Gr)$ is
  $\E_n$-monoidal and preserves colimits.
  The functor
  \[ \Pro(\EE_+^\Gr) \to \Pro((\EE_+^\Gr)^{\mathrm{thin}}) \]
  is the left adjoint of an $\E_n$-monoidal functor and is therefore oplax monoidal.\footnote{It is easier to think in terms of opposite categories here, since they are presentable and then this a lax monoidal right adjoint.}
  In order to show that $e_+$ is actually $\E_n$-monoidal we use the fact that $e_+$ is $\E_n$-monoidal after restricting to thin objects and \Cref{lem:sum-prod} which lets us convert between sums and product.
  \begin{align*}
    e &(X \otimes Y)
    \simeq e\left( \left(\bigoplus_{j \geq 0} X_k(j) \right) \otimes \left( \bigoplus_{k \geq 0} Y_k(k) \right) \right) 
    \simeq e\left( \bigoplus_{j,k \geq 0}  X_k(j) \otimes Y_k(k) \right) \\
    &\simeq \bigoplus_{j,k \geq 0} e\left( X_k(j) \otimes Y_k(k) \right) 
    \simeq \bigoplus_{j,k \geq 0} e(X_k(j)) \otimes e(Y_k(k)) 
    \simeq \prod_{j,k \geq 0} e(X_k(j)) \otimes e(Y_k(k)) \\
    &\simeq \left( \prod_{j \geq 0 }e(X_k(j)) \right) \otimes \left( \prod_{k \geq 0 } e(Y_k(k)) \right) 
    \simeq \left( \bigoplus_{j \geq 0 }e(X_k(j)) \right) \otimes \left( \bigoplus_{k \geq 0 } e(Y_k(k)) \right) \\
    &\simeq  e \left( \bigoplus_{j \geq 0 } X_k(j) \right) \otimes e \left( \bigoplus_{k \geq 0 } Y_k(k) \right)
  \end{align*}
  
  The key step in these manipulations is the point where we converted the sum over $j,k \geq 0$ into a product over $j,k \geq 0$ and it is at this point that we used our restriction to the positively graded setting.
  Note also that in pro-objects it is products and not sums which distribute over the tensor product.
  The manipulations used to prove that $e$ (and therefore $e_+$) is fully faithful are similar.
  \begin{center}
    \begin{tikzcd}[row sep=small]
      \Map_{\EE^\Gr}\left( \bigoplus_j X_j(j) , \bigoplus_k Y_k(k) \right) \ar[r] \ar[dd, "\simeq"] &
      \Map_{\Pro(\EE)^\Gr}\left( e\left( \bigoplus_j X_j(j) \right), e\left( \bigoplus_k Y_k(k) \right) \right) \ar[d, "\simeq"] \\
       & \Map_{\Pro(\EE)^\Gr}\left(  \bigoplus_j e(X_j(j)), \bigoplus_k e(Y_k(k)) \right) \ar[d, "\simeq"] \\
      \Map_{\EE^\Gr}\left( \bigoplus_j X_j(j) , \prod_k Y_k(k) \right) \ar[r] \ar[d, "\simeq"] &
      \Map_{\Pro(\EE)^\Gr}\left( \bigoplus_j e(X_j(j)), \prod_k e(Y_k(k))  \right) \ar[d, "\simeq"] \\
      \prod_{j,k} \Map_{\EE^\Gr}\left( X_j(j) , Y_k(k) \right) \ar[r] \ar[d, "\simeq"] &
      \prod_{j,k} \Map_{\Pro(\EE)^\Gr}\left( e(X_j(j)), e(Y_k(k))  \right) \ar[d, "\simeq"] \\
      \prod_i \Map_{\EE}\left( X_i , Y_i \right) \ar[r, "\simeq"] &
      \prod_i \Map_{\EE}\left( X_i, Y_i \right)
    \end{tikzcd}
  \end{center}  
\end{proof}

\begin{lem} \label{lem:pro-thin-small}
  $\Bar$ and $\Cobar$ each send objects in the image of $e_+$ to objects in the image of $e_+$.
\end{lem}

\begin{proof}
  For $\Bar$ this follows from \Cref{lem:pro-thin-good}
  since the bar construction is composed of tensor products and a colimit.
  For $\Cobar$ this follows from the fact that because we are in the positively graded setting
  the tot tower in \Cref{cnstr:partial-tor-res}
  is finite in each degree and therefore $\Cobar(C)$ is a constant pro-object in each degree
  (i.e. it is in the image of $e_+$).
\end{proof}

\begin{lem} \label{lem:compare-bar}
  Suppose we are given a monoidal left adjoint
  \[ f: ((\EE_1)_+^\Gr)^{\mathrm{thin}} \to ((\EE_2)_+^\Gr)^{\mathrm{thin}}, \]
  then there is a (vertically) right adjointable square
  \begin{center}
    \begin{tikzcd}
      \Alg((\EE_1)_+^\Gr) \ar[r, "\wt{f}"] \ar[d, "\Bar"] & 
      \Alg((\EE_2)_+^\Gr) \ar[d, "\Bar"] \\
      \cAlg((\EE_1)_+^\Gr) \ar[r, "\wt{f}"] &
      \cAlg((\EE_2)_+^\Gr).
    \end{tikzcd}
  \end{center}  
\end{lem}

\begin{proof}
  Kan extending $f$ to all of $\EE_+^\Gr$ we obtain a monoidal left adjoint,
  \[ \wt{f} : (\EE_1)_+^\Gr \to (\EE_2)_+^\Gr. \]
  From \cite[5.2.3.11]{HA} we now obtain the desired square.
  In order to show this square is right adjointable we use
  the embedding into pro-thin objects of \Cref{cnstr:pro-thin-embed}.
  
  The functor $\Pro(f) : \Pro(\EE_1)_+^\Gr \to \Pro(\EE_2)_+^\Gr$
  is an $\E_n$-monoidal left adjoint which preserves all limits.
  Consequently we can apply \cite[5.2.3.11]{HA}
  to obtain a (vertically) right adjointable square
  \begin{center}
    \begin{tikzcd}
      \Alg(\Pro(\EE_1)_+^\Gr) \ar[r, "\Pro(f)"] \ar[d, "\Bar"] & 
      \Alg(\Pro(\EE_2)_+^\Gr) \ar[d, "\Bar"] \\
      \cAlg(\Pro(\EE_1)_+^\Gr) \ar[r, "\Pro(f)"] &
      \cAlg(\Pro(\EE_2)_+^\Gr).
    \end{tikzcd}
  \end{center}  
  Using \Cref{lem:pro-thin-good} and \cite[5.2.3.11]{HA} we can extend this square to a cube
  via the colimit preserving, fully faithful embeddings into pro-thin objects.
  At this point right adjointability follows from \Cref{lem:pro-thin-small}
  which says that $\Cobar$ and $\Bar$ send objects in the image of $e_+$ to objects in the image of $e_+$.
\end{proof}

\begin{lem} \label{lem:basic-computation}
  In $\Sp^\Gr$ the unit map
  $\o\{\Ss^0(1)\} \to \Cobar (\Bar (\o\{ \Ss^0(1)\} ))$
  is an equivalence.
\end{lem}

\begin{proof}
  % It is easy to check that the unit map is an equivalence in degrees $0$ and $1$,
  % therefore it suffices to argue that $\Cobar (\Bar (\o\{ \Ss^0(1)\} ))$ is
  % a free algebra on a class in degree $1$.
  
  Applying \Cref{lem:compare-bar} to the map $\Sp \to \Mod(\Z)$ and using the fact that 
  the bar and cobar produce objects which are levelwise finite (see \Cref{cnstr:partial-tor-res})
  it suffices to observe that the unit is an equivalence in the graded $\Z$-linear case
  where this is the usual Koszul duality between polynomial and exterior algebras.
\end{proof}

\begin{lem} \label{lem:free-unit-eq}
  Given an $X \in \EE^\Gr$ concentrated in positive degrees
  the unit map
  $\o\{ X \} \to \Cobar (\Bar (\o \{ X \} ))$
  is an equivalence.
\end{lem}

\begin{proof}
  We begin by handling the case where $X$ is thin.
  Using the fact that $\Sp_{\geq 0}^\Gr$ is the free $\E_1$-monoidal category on an object $\Ss^0(1)$ we can construct a monoidal left adjoint
  \[ \wt{f}^* : \Sp_{\geq 0}^\Gr \to \EE_{\geq 0} \]
  which sends $\Ss^0(1)$ to $X$ whose right adjoint $\wt{f}_*$ sends $Y \in \EE_{\geq 0}^\Gr$ to
  $ \left( k \mapsto \Map_{\EE^\Gr}^\Sp( X^{\otimes k}, Y) \right) $.  
  Since $X$ is thin and concentrated in positive degrees
  $\wt{f}^*$ and $\wt{f}_*$ restrict to an adjunction
  \[ f^* : (\Sp_+^\Gr)^{\mathrm{thin}} \rightleftharpoons (\EE^\Gr_+)^{\mathrm{thin}} : f_*. \]

  Since $f^*$ is a monoidal left adjoint we have $f^* (\o\{Y\}) \simeq \o\{ f^*Y \}$.
  \Cref{lem:compare-bar} now allows us to reduce showing that the unit map
  $\o\{X \} \to \Cobar ( \Bar( \o \{ X \}))$
  is an equivalence to
  \Cref{lem:basic-computation}.
  The general case now follows by writing $X$ as a filtered colimit of thin objects and appealing to \Cref{lem:cobar-sifted}.
\end{proof}

\begin{proof}[Proof (of \Cref{thm:En-bar-cobar}).]
  We proceed by induction on $n$ with $n=1$
  as our base case.

  In order to prove that $\Bar : \Alg(\EE_+^\Gr) \to \cAlg(\EE_+^\Gr)$
  is an equivalence we will show that it is fully faithful and its right adjoint is conservative.
  Recall that we already proved $\Cobar$ is conservative in \Cref{cor:agree-through-k}.
  We prove $\Bar$ is fully faithful by showing that the unit map
  \[ A \to \Cobar ( \Bar ( A )) \]
  is an equivalence for every $A \in \Alg(\EE_+^\Gr)$.
  From \Cref{lem:free-unit-eq} we know the unit is an equivalence when $A$ is a free algebra.
  We also know that both $\Bar$ and $\Cobar$ commute wtih geometric realizations (see \Cref{lem:cobar-sifted}), therefore it suffices to argue that every $A$ can be written as a geometric realization of a simplicial diagram of free algebras. 
  This, in turn, follows from the fact that the free--underlying adjunction on $\Alg(\EE_+^\Gr)$ is monadic.

  Now we handle the inductive step.
  Applying \cite[5.2.3.12 and 5.2.3.14]{HA} we are able to construct the following diagram
  where we have indicated the maps which are equivalences based on our inductive hypothesis.
  \begin{center}    
    \begin{tikzcd}[column sep = small]
      {\Alg_{/\E_{a+b}}(\EE_+^\Gr)} \ar[rr, "{\Bar^{(a+b)}}"] \ar[d, "\simeq"] & &
      {\cAlg_{/\E_{a+b}}(\EE_+^\Gr)} \ar[d, "\simeq"] \\
      {\Alg_{/\E_a} \left( \Alg_{\E_b/\E_{a+b}} (\EE_+^\Gr) \right) } \ar[r, "{\wt{\Bar}^{(a)}}"] \ar[d] &
      \cAlg_{/\E_a}\left(\Alg_{\E_b/\E_{a+b}}(\EE_+^\Gr)\right) \ar[r, "\Bar^{(b)}", "\simeq"'] \ar[d] &
      \cAlg_{/\E_a}\left(\cAlg_{\E_b/\E_{a+b}}(\EE_+^\Gr)\right) \\
      \Alg_{/\E_a}(\EE_+^\Gr) \ar[r, "\Bar^{(a)}", "\simeq"'] &
      \cAlg_{/\E_a}(\EE_+^\Gr) 
    \end{tikzcd}
  \end{center}

  In order to complete the proof we must show that $\wt{\Bar}^{(a)}$ is an equivalence.
  The underlying object functor $\Alg_{\E_b/\E_{a+b}}(\EE_+^\Gr) \to \EE_+^\Gr $ preserves both limits and geometric realizations, therefore by \cite[5.2.3.11]{HA} the bottom square is (horizontally) right adjointable.  
  The vertical arrows in the square are conservative,
  therefore the (co)unit map of the $\wt{\Bar}^{(a)}$--$\wt{\Cobar}^{(a)}$ adjunction is an equivalence at an object $X$ iff it is an equivalence on underlying. This follows from the inductive assumption that $\Bar^{(a)}$ is an equivalence.
\end{proof}

\subsection{Using bar-cobar duality}
\label{subsec:uses}\ 

We end the appendix with a couple lemmas focused on exposing the consequences of bar-cobar duality necessary for \Cref{sec:budget}.

\begin{lem} \label{lem:cofree-formula} \label{lem:cofree-to-sqz}
  The underlying object functor
  $ \cAlg_{\E_n}(\EE_+^\Gr) \to \EE_+^\Gr$
  has a right adjoint
  \[ \mathrm{coFree} : \EE_+^\Gr \to \cAlg_{\E_n}(\EE_+^\Gr) \]
  whose composite with $\Cobarn$ sends an object $X(k)$ to the square-zero algebra
  $\o \oplus \Sigma^{-n}X(k)$.
  The underlying object of $\mathrm{coFree}(Y)$ is given by $\prod_k D_k^{c\E_n}(Y)$.
\end{lem}

\begin{proof}
  $\Pro(\EE)_+^\Gr$
  is the opposite of a presentably $\E_n$-monoidal category and therefore by \cite[3.1.3.13]{HA}
  we have an underlying--cofree adjunction
  \[ \cAlg_{\E_n}\left( \Pro(\EE)_+^\Gr \right) \rightleftharpoons \Pro(\EE)_+^\Gr : \widehat{\mathrm{coFree}} \]
  where $\widehat{\mathrm{coFree}}(X)$ is given by $\prod_k D_k^{c\E_m}(X)$.

  The underlying object functors commutes with the $\E_n$-monoidal embedding $e_+$ of \Cref{cnstr:pro-thin-embed} giving us a square
  \begin{center}
    \begin{tikzcd}
      \cAlg_{\E_n}\left( \EE_+^\Gr \right) \ar[r] \ar[d , "e_+"] &
      \EE_+^\Gr \ar[d, "e_+"] \\
      \cAlg_{\E_n}\left( \Pro(\EE)_+^\Gr \right) \ar[r] &
      \Pro(\EE)_+^\Gr.
    \end{tikzcd}
  \end{center}
  As the two vertical maps are fully faithful,
  the underlying object functor will admit a right adjoint
  if the cofree $\E_n$-coalgebra functor on $\Pro(\EE)_+^\Gr$ sends objects in the image of $e_+$ to objects in the image of $e_+$.
  In order to check this we observe that
  if $X$ is concentrated in positive degrees, then the infinite product $\prod_k D_k^{c\E_n}(e(X))$
  is a finite product in any fixed degree and the limit used to compute $D_k^{c\E_n}(e(X))$ is finite.

  In order to compute $\Cobarn( \mathrm{coFree}(X(k)))$ we again observe that it we compute what this object is in pro-thin objects and the result is in the image of $e_+$, then this is the correct answer.
  In pro-thin objects we can use \cite[5.2.3.15]{HA} to conclude that the underlying object of
  $ \Cobarn (\widehat{\mathrm{coFree}}(e(X(k)))) $ is $\o \oplus \Sigma^{-n}e(X(k))$.
  In particular, this is in the image of $e_+$ as desired.

  In order to identify the algebra structure on $\Cobarn(\mathrm{coFree}(X(k)))$ 
  we observe that all augmented $\E_n$-algebras with underlying object
  $\o \oplus \Sigma^{-n}e(X(k))$
  are equivalent, 
  as they can all be obtained by truncating a free algebra to lie in degrees $\leq k$.
\end{proof}

% \begin{cor} \label{cor:En-agree-through-k}
%   A map of postively graded $\E_n$-algebras $A \to B$ in $\EE$ is an equivalence through degree $k$
%   iff the map of $\E_n$-coalgebras $\Barn(A) \to \Barn(B)$ is an equivalence through degree $k$.

%   Dually, A map of postively graded $\E_n$-coalgebras $C \to D$ in $\EE$ is an equivalence through degree $k$ iff the map of $\E_n$-algebras $\Cobarn(C) \to \Cobarn(D)$ is an equivalence through degree $k$.
% \end{cor}

% \begin{proof}
%   From \Cref{lem:iterated} we know $\Barn$ (resp. $\Cobarn$) is computed by an iterated bar (cobar) construction. The corollary then follows from \Cref{cor:agree-through-k} which is the $n=1$ case.
% \end{proof}

\begin{cnstr} \label{cnstr:free-res}
  Given a postively graded $\E_n$-algebra $R$ in $\EE$
  we can inductively produce a filtration
  \[ \o_{\EE^{\Gr}} = R^0 \xrightarrow{r_1} R^1 \xrightarrow{r_2} R^2 \to \cdots \to R \]
  converging to $R$ such that
  \begin{enumerate}
  \item the map $R^{k} \to R$ is an equivalence through degree $k$ and
  \item the map $r_{k}$ fits into a pushout square of positively graded $\E_n$-algebras
    \begin{center}
      \begin{tikzcd}
        \o_{\EE^{\Gr}}\left\{ X^k(k) \right\} \ar[r, "\mathrm{aug}"] \ar[d, "s_{k}"] & \o_{\EE^{\Gr}} \ar[d] \\
        R^{k-1} \ar[r, "r_{k}"] & R^{k}. \pushout
      \end{tikzcd}
    \end{center}
    for some object $X^k \in \EE$.
  \end{enumerate}  
  via the following procedure:
  Given $R^{k-1}$ we let $X^k \coloneqq \fib((R^{k-1})_k \to (R^k)_k)$
  this choice of $X^k$ naturally comes equipped with a map
  $X^k \to (R^{k-1})_k$ and a nullhomotopy of the composite $X^k \to (R^{k-1})_k \to (R^k)_k$
  which allows us to define the map $s_k$ and the factorization of the map $R^{k-1} \to R$ through $r_k$.
  Since the free algebra we used started in grading $k$ we have a cofiber sequence
  \[ X^k \to (R^{k-1})_{k} \to (R^{k})_k \]
  from which (1) follows.
  Convergence of the $R^k$ to $R$ follows from condition (1).  
  \tqed
\end{cnstr}

\begin{lem} \label{lem:free-res-terms}
  Let $R$ be a positively graded $\E_n$-algebra in $\EE$.
  We can identify the object $X^k$ from \Cref{cnstr:free-res} with $\Sigma^{-1-n}\Barn(R)_k$.
\end{lem}

\begin{proof}
  Upon applying $\Barn$ to the resolution of \Cref{cnstr:free-res} we obtain a
  resolution of $\Barn(R)$ which at its $k^{\mathrm{th}}$ term is a 
  pushout under the square-zero $\E_n$-coalgebra $\o \oplus \Sigma^n X^k(k)$
  (see \Cref{lem:cofree-to-sqz}).    
  Since pushouts of coalgebras are computed on underlying, this allows us to read off that
  \[ \Barn(R)_k \simeq \Sigma^{n+1} X^k. \qedhere \]
\end{proof}

%%% Local Variables:
%%% mode: latex
%%% TeX-master: multmoore.tex
%%% End: